\documentclass[10pt]{amsart}
\usepackage{amsmath,amsthm,amsfonts,amssymb,amscd,flafter,pinlabel}
\usepackage{mathtools}
\usepackage[mathscr]{eucal}
\usepackage{graphics}
 \usepackage[all]{xy}
\usepackage{caption, subcaption}
\usepackage[usenames,dvipsnames]{color}
\usepackage{graphicx}

\usepackage{pgf}
\usepackage{tikz}
\usetikzlibrary{arrows,automata}
\usepackage[latin1]{inputenc}

\usepackage[colorlinks=true, urlcolor=NavyBlue, linkcolor=NavyBlue, citecolor=NavyBlue, pdftitle={Stein trisections and homotopy 4-balls}, pdfauthor={Peter Lambert-Cole}, pdfsubject={symplectic}, pdfkeywords={4-manifolds}]{hyperref}

\newtheorem{theorem}{Theorem}[section]
\newtheorem{corollary}[theorem]{Corollary}
\newtheorem{conjecture}[theorem]{Conjecture}
\newtheorem{question}[theorem]{Question}
\newtheorem{proposition}[theorem]{Proposition}

\newtheorem{lemma}[theorem]{Lemma}

\theoremstyle{definition}
\newtheorem{definition}[theorem]{Definition}
\newtheorem {example}[theorem]{Example}

\theoremstyle{definition}
\newtheorem{remark}[theorem]{Remark}

\newcommand{\CC}{\mathbb{C}}
\newcommand{\RR}{\mathbb{R}}
\newcommand{\del}{\partial}

\newcommand{\ZZ}{\mathbb{Z}}

\newcommand{\cD}{\mathcal{D}}

\newcommand{\cT}{\mathcal{T}}

\newcommand{\cC}{\mathcal{C}}
\newcommand{\cK}{\mathcal{K}}
\newcommand{\cS}{\mathcal{S}}

\newcommand{\cG}{\mathcal{G}}

\newcommand{\DD}{\mathbb{D}}

\begin{document}

\title[{Stein trisections and homotopy 4-balls}]{Stein trisections and homotopy 4-balls}

\author[P. Lambert-Cole]{Peter Lambert-Cole}
\address{ University of Georgia}
\email{plc@uga.edu}

\keywords{4-manifolds, trisections}
\subjclass[2010]{57R17; 53D05}
\maketitle


\begin{abstract}

A homotopy 4-ball is a smooth 4-manifold with boundary $S^3$ that is homotopy-equivalent to the standard $B^4$.  The smooth 4-dimensional Schoenflies problem asks whether every homotopy 4-ball in $S^4$ (or equivalently $\CC^2$) is standard.  It is well-known that if a homotopy 4-ball embeds as a compact, pseudoconvex domain in a Stein surface, then it must be standard.  In this paper, we describe a compelling reimbedding construction for homotopy 4-balls in $\CC^2$.  In particular, given a homotopy 4-ball in $\CC^2$, we construct a diffeomorphic domain that is the union of three pseudoconvex domains.  Moreover, we give an analytic criterion that ensures this domain is a standard 4-ball.

\end{abstract}

\section{Introduction}

A {\it homotopy 4-ball} is a compact, contractible smooth 4-manifold with boundary homeomorphic to $S^3$.  By Freedman, every homotopy 4-ball $B$ is homeomorphic to the {\it standard 4-ball} $B^4 \subset \RR^4$ and it remains an open question whether there exist {\it exotic} homotopy 4-balls that are not diffeomorphic to the standard 4-ball.  The related smooth, 4-dimensional Schoenflies problem asks whether every homotopy 4-ball that smoothly embeds in $S^4$ is standard.

It is well-known that the smooth 4-dimensional Schoenflies problem can be rephrased in terms of the complex geometry of $\CC^2$.  Let $\Omega \subset \CC^n$ be a domain.  An upper semicontinuous function $f: \Omega \rightarrow \RR$ is {\it plurisubharmonic} if its restriction to every complex curve in $\Omega$ is subharmonic.  A domain $\Omega$ is {\it pseudoconvex} if it admits a continuous plurisubharmonic exhaustion function. The method of filling by holomorphic disks can be used to determine the smooth topology of pseudoconvex homotopy 4-balls.

\begin{theorem}[Gromov \cite{Gromov}, Eliashberg \cite{Eliashberg}]
\label{thrm:pseudoconvex-standard}
Let $\Omega \subset \CC^2$ be a pseudoconvex homotopy 4-ball.  Then $\Omega$ is diffeomorphic to the standard 4-ball.
\end{theorem}

In particular, the smooth 4-dimensional Schoenflies conjecture is true if every smoothly-embedded homotopy 4-ball in $\RR^4 \cong \CC^2$ can be made pseudoconvex.

The main result of this paper is to construct reimbeddings of homotopy 4-balls in $\CC^2$ that are close to being pseudoconvex.  

\begin{theorem}
\label{thrm:pseudoconvex-reimbedding}
Let $S$ be a smoothly embedded $S^3$ in $\CC^2$ that bounds a homotopy 4-ball $B$.  Then there exists domain $Z = Z_1 \cup Z_2 \cup Z_3$ such that $Z$ is diffeomorphic to $B$ and each $Z_{\lambda}$ is pseudoconvex. 
\end{theorem}

Of course, the union of pseudoconvex domains is not necessarily pseudoconvex.  In particular, we cannot simply apply Theorem \ref{thrm:pseudoconvex-standard} to deduce that $B \cong Z$ is a standard homotopy 4-ball. However, there are natural ways to extend a domain in $\CC^n$ to a pseudoconvex domain, either by taking its hull with respect to polynomial, rational, or plurisubharmonic functions or by taking its envelope of holomorphy.

\pagebreak

\begin{question}
\label{q:hulls-envelope}
Let $Z$ be the domain constructed in Theorem \ref{thrm:pseudoconvex-reimbedding}.
\begin{enumerate}
    \item What are the polynomial/rational/plurisubharmonic hulls of $Z$?
    \item What is the envelope of holomorphy of $Z$?
\end{enumerate}
\end{question}

We can give an analytic criterion that ensures $Z$ is pseudoconvex.  Morally speaking, the domain $Z$ is an attempt to reconstruct $B$ as an analytic polyhedron.  The boundary $\del Z$ is stratified and the top-dimensional strata are Levi-flat and foliated by holomorphic disks.  The codimension-1 {\it trisection spine} $\cS = V_1 \cup V_2 \cup V_3$ (see Section \ref{sec:reconstruction}) plays the role of the {\it spine} or {\it arete} of the analytic polyhedron.  If $Z$ is an analytic polyhedron, then every holomorphic function on $Z$ reaches its maximum modulus along this spine.  

Here, the converse implies that $Z$ is pseudoconvex and therefore a standard 4-ball.  Specifically, let $Z \subset \Omega \subset \CC^2$ be a domain.  For every holomorphic function $f$ on $\Omega$ and compact set $K \subset \Omega$, let $|f|_K$ denote the maximum modulus that $f$ attains on $K$.  The {\it Shilov boundary} of $Z$ in $\Omega$ is the smallest closed subset $K$ such that $|f|_Z = |f|_K$ for all holomorphic functions on $\Omega$.

\begin{theorem}
\label{thrm:Shilov}
Suppose that the trisection spine $\cS$ is the Shilov boundary for $Z$ in $\Omega$.  Then $Z$ is pseudoconvex and diffeomorphic to the standard 4-ball.
\end{theorem}

It would be very interesting to give topological criteria in terms of a trisection diagram that ensure $\cS$ is the Shilov boundary of $Z$.

\subsection{Stein trisections}

The proof of Theorem \ref{thrm:pseudoconvex-reimbedding} relies on new results on trisections of 4-manifolds, especially of Stein trisections.  

\begin{definition}
A {\it trisection} of a closed 4-manifold is a decomposition $X = X_1 \cup X_2 \cup X_3$ such that
\begin{enumerate}
    \item each $X_{\lambda}$ is a diffeomorphic to a 4-dimensional 1-handlebody with $b_1 = k_{\lambda} \geq 0$.
    \item each double intersection $H_{\lambda} = X_{\lambda} \cap X_{\lambda-1}$ is a 3-dimensional 1-handlebody of genus $g \geq 0$, and
    \item the triple intersection $\Sigma = X_1 \cap X_2 \cap X_3$ is a closed surface of genus $g$.
\end{enumerate}
Furthermore, a {\it relative trisection} of a compact 4-manifold with boundary is a decomposition $X = X_1 \cup X_2 \cup X_3$ satisfying the first two criteria above as well as
\begin{enumerate}
    \item the triple intersection $\Sigma$ is a properly embedded surface $\Sigma$ of genus $g$ and $p \geq 1$ boundary components
    \item the boundary $\del H_{\lambda}$ is a compact surface of genus $g' \leq g$ and $p$ boundary components
    \item the intersection $X_{\lambda} \cap \del X$ is homeomorphic to $\Sigma_{g,p} \times [0,1]$.
\end{enumerate}
\end{definition}

The {\it spine} of a (relative) trisection is the union $H_1 \cup H_2 \cup H_3$ of the 3-dimensional strata.  The well-known result of Laudenbach-Po\'enaru \cite{LP} implies that the 4-dimensional pieces can be glued in uniquely, up to diffeomorphism.  Therefore the spine of a trisection determines the diffeomorphism type of the 4-manifold.

The standard 4-ball admtis a unique relative trisection $\cT_0$ where the central surface is a disk.  Furthermore, there is a (interior) stabilization operation that increases both the genus of the central surface of a (relative) trisection and some $k_{\lambda}$ by one.  We say that a relative trisection of $B^4$ is {\it standard} if it can be obtained from $\cT_0$ by a sequence of interior stabilizations.

Following Meier--Schirmer--Zupan \cite{MSZ} and by analogy with Heegaard splittings, we say that a (relative) trisection is {\it reducible} if there exists a simple closed curve $\delta \subset \Sigma$ that bounds a disk $V_{\lambda} \subset H_{\lambda}$ for each $\lambda= 1,2,3$.  Consequently, each pair $V_{\lambda} \cup V_{\lambda+1}$ is an embedded 2-sphere in $\del X_{\lambda}$, which bounds a 3-ball $B_{\lambda} \subset X_{\lambda}$.  The union $S = B_1 \cup B_2 \cup B_3$ is a smoothly embedded 3-sphere that unique defined up to diffeomorphism.  Conversely, we say that a (relative) trisection is {\it reducible} with respect to $S \subset X$ if the trisection stratifies $S$ into a triple of 3-balls as above.

\begin{theorem}
Let $S \subset X$ be a smoothly embedded $S^3$.  Then
\begin{enumerate}
    \item there exists a (relative) trisection of $X$ that is reducible with respect to $S$
    \item if $\cT$ is any (relative) trisection of $X$, it can be (interior) stabilized a finite number of times to obtain a (relative) trisection that is reducible with respect to $S$.
\end{enumerate}
\end{theorem}

Finally, let $\Omega$ be a domain in a complex surface.  Recall that $P \subset \Omega$ is an {\it analytic polyhedron} if there exists holomorphic functions $f_1,\dots,f_k: \Omega \rightarrow \CC$ such that
\[P = \{z \in \Omega: |f_1(z)|, \dots, |f_k(z)| \leq 1\}\]
In particular, analytic polyhedra are pseudoconvex and therefore Stein domains.

\begin{definition}
Let $Z$ be a compact domain in a complex surface.  A {\it Stein trisection} of $Z$ is a (relative) trisection decomposition $Z = Z_1 \cup Z_2 \cup Z_3$ such that each $Z_{\lambda}$ is an analytic polyhedron.  Let $\cT$ denote the induced smooth trisection of $Z$.  In this case we say that $\cT$ {\it admits a Stein structure}.
\end{definition}

The notion of a {\it Stein trisection} of a complex surface was introduced in \cite{LM-Rational}, where it was shown that the standard trisection of $\mathbb{CP}^2$ is Stein \cite{LM-Rational}.  Zupan has recently constructed Stein trisections of hypersurfaces in $\mathbb{CP}^3$ of every degree.  The corresponding notion of a {\it Weinstein trisection} of a closed, symplectic 4-manifold was also introduced in \cite{LM-Rational}.  Every symplectic 4-manifold admits a Weinstein trisection \cite{LMS} and this construction was used to give a novel interpretation of the Symplectic Thom conjecture and the adjunction inequality in \cite{SympThom}.

The main trisections result of this paper is the following construction.

\begin{theorem}
Every standard relative trisection of $B^4$ admits a Stein structure.
\end{theorem}

Reverting back to the study of Schoenflies 4-balls in $\CC^2$, we have the following immediate corollary of the previous two theorems.

\begin{corollary}
Let $B \subset B^4 \subset \CC^2$ be a homotopy 4-ball bounded by a smoothly embedded $S^3$.  
\begin{enumerate}
\item there exists a relative trisection of $B^4$ that is reducible with respect to $\del B$,
\item there exists a relative Stein trisection of $B^4$ that is reducible with respect to $\del B$.
\end{enumerate}
\end{corollary}

\subsection{Reconstruction}

Using this result, we can describe the reimbedding construction of Theorem \ref{thrm:pseudoconvex-reimbedding}.  Let $B \subset B^4 \subset \CC^2$ be a homotopy 4-ball with smooth boundary.  We can choose a Stein trisection $B^4 = X_1 \cup X_2 \cup X_3$ that is reducible with respect to $\del B$.  The boundary $\del B$ intersects each sector along a 3-ball.  Applying a result of Bedford-Klingenberg \cite {BK-filling}, we can fill by holomorphic disks to replace by a Levi-flat 3-ball $D_{\lambda}$ that is foliated by holomorphic disks.  This determines a new pseudoconvex sector $Z_{\lambda}$ that is diffeomorphic to $X_{\lambda} \cap B$.  In particular, the union $Z = Z_1 \cup Z_2 \cup Z_3$ is a trisected 4-manifold with boundary $S^3$.  The trisected 4-manifolds $B$ and $Z$ have the exact same spines, hence they are diffeomorphic.

This reimbedding depends only on the existence of a Stein trisection, as the result of Bedford and Klingenberg applies in all Stein surfaces.  Work in progress by the author with Islambouli, Meier and Starkston constructs Stein trisections on all compact Stein surfaces.  Consequently, this reimbedding result should extend to all homotopy 4-balls in compact Stein surfaces.  We therefore make the following conjecture, which extends the smooth 4-dimensional Schoenflies conjecture and can be viewed as analogous to a conjecture of Gompf on minimal symplectic 4-manifolds  \cite{Gompf}.

\begin{conjecture}
Let $X$ be a compact Stein domain of complex dimension 2 and let $B \subset X$ be an embedded homotopy 4-ball bounded by a smoothly embedded $S^3$.  Then $B$ is a standard 4-ball.
\end{conjecture}

\subsection{Acknowledgements}

I would like to thank David Gay, Gordana Matic, Jeff Meier, Laura Starkston and Alex Zupan for helpful conversations.  This work is partially supported by NSF grant DSM-1664567.

\section{Trisections}
\label{sec:trisections}

Trisections are 4-dimensional generalization of Heegaard splittings of 3-manifolds.  In this section, we will describe the basic terminology and background results of the theory.  At this point, there is quite a voluminous literature defining the basic building blocks of the theory:
\begin{enumerate}
    \item trisections of 4-manifolds \cite{Gay-Kirby}
    \item relative trisections of 4-manifolds with boundary \cite{CGP1,CGP2}
    \item bridge trisections of surfaces \cite{MZ1,MZ2}
    \item constructing branched coverings via surfaces in bridge position \cite{LM-Rational,LMS}
    \item bridge trisections for properly embedded surfaces in 4-manifolds with boundary \cite{Meier-Relative}
\end{enumerate}

The two key ideas we exploit in this paper are:
\begin{enumerate}
    \item Trisections pull back under branched coverings when the branch locus is in bridge position.  This has been used to construct trisections of projective surfaces \cite{LM-Rational} and symplectic 4-manifolds \cite{LMS} and will be used in Section \ref{sec:Stein} to construct Stein trisections.
    \item In a simple branched covering, an (interior) stabilization of the branch locus downstairs results in an (interior) stabilization of the trisection upstairs.
\end{enumerate}

\subsection{Trisections}

Let $X$ be a closed, smooth 4-manifold.

\begin{definition}
A $(g;k_1,k_2,k_3)$-{\it trisection} $\cT$ of $X$ is a decomposition $X = Z_1 \cup Z_2 \cup Z_3$ such that for each $\lambda \in \ZZ_3$,
\begin{enumerate}
\item Each $Z_{\lambda}$ is diffeomorphic to $\natural_{k_{\lambda}} (S^1 \times B^3)$ for $k_{\lambda} \geq 0$,
\item Each double intersection $H_{\lambda} = Z_{\lambda-1} \cap Z_{\lambda}$ is a 3-dimensional, genus $g$ handlebody, and
\item the triple intersection $\Sigma = Z_1 \cap Z_2 \cap Z_3$ is a closed, genus $g$ surface.
\end{enumerate}
We refer to $\Sigma$ as the {\it central surface} of the trisection, the genus $g$ of $\Sigma$ as the {\it genus} of the trisection, and the union $H_1 \cup H_2 \cup H_3$ as the {\it spine} of the trisection.
\end{definition}

A trisection of $X$ can be encoded by a {\it trisection diagram} that describes how each $H_{\lambda}$ is attached to $\Sigma$ \cite[Definition 3]{Gay-Kirby}.  In particular, since $H_{\lambda}$ is a 1-handlebody, the gluing is determined by how a complete system of compressing disks is attached to $\Sigma$.  These compressing disks are attached along a cut system of curves in $\Sigma$ and these curves can be depicted on $\Sigma$ itself.

\subsection{Trisections of spheres}

The $n$-sphere has a standard trisection, obtained by pulling back a trisection on the unit disk $D$ in $\mathbb{R}^2$.  

\begin{definition}
The {\it standard trisection} $D = D^s_1 \cup D^s_2 \cup D^s_3 \subset \RR^2$ of the unit disk is defined by choosing the following subsets (in radial coordinates):
\begin{align*}
D^s_1 &= \left\{ 0 \leq \theta \leq \frac{2\pi}{3} \right\}, & D^s_2 &= \left\{\frac{2\pi}{3} \leq \theta \leq \frac{4\pi}{3} \right\}, &  D^s_3 &= \left\{ \frac{4\pi}{3} \leq \theta \leq 2 \pi \right\}.
\end{align*}
\end{definition}
It is symmetric under rotation by $2\pi/3$.

\begin{definition}[\cite{LC-Miller}]
\label{def:standard-tri-sphere}
The {\it standard trisection} of $S^{k}$ for $k\ge 2$ 
is the decomposition $\cT_{std} = \{\pi^{-1}(D^s_i) \cap S^{k}\}$ where $\pi: \mathbb{R}^{k+1} \rightarrow \mathbb{R}^2$ is a coordinate projection and $D^s_1 \cup D^s_2 \cup D^s_3$ is the standard trisection of the unit disk in $\mathbb{R}^2$.
\end{definition}

\begin{example}
\textit{The standard trisection of $S^2$}.  The triple intersection of $\cT_{std}$ consists of the North and South poles; the double intersections are three meridians connecting the poles; and the sectors are the three bigons bounded by the pairs of meridians.
\end{example}

\begin{example}
\label{ex:std-tri-s3}
\textit{The standard trisection of $S^3$}.  The triple intersection consists of a single $S^1$, unknotted in $S^3$; each double intersection is a $D^2$ Seifert disk for the unknot; and each sector of the trisection is a 3-ball.  In addition, we can give two other useful interpretations:
\begin{enumerate}
\item The lens space $L(1,3)$.  Here, $S^3$ is decomposed into three lenses: i.e take $D^2 \times [0,1]$ and collapse the vertical boundary to a circle.  These lenses are glued sequentially along their boundaries.
\item The genus 0 open book decomposition of $S^3$. Recall that $S^3$ has an open book decomposition $(B,\phi)$ where the binding $B$ is the unknot and the fibration $\phi: S^3 \smallsetminus \nu(B) \cong S^1 \times D^2 \rightarrow S^1$ is trivial.  This induces a trisection as follows.  Choose any three distinct pages $\left\{\phi^{-1}(\theta_1),\phi^{-1}(\theta_2), \phi^{-1}(\theta_3)\right\}$ of the open book decomposition.  Their union becomes the spine of the trisection and the binding $B$ is the codimension 2 stratum of the trisection.  The complement of the spine consists of three open 3-balls.
\end{enumerate}
\end{example}

The following lemma is immediate from the definition.

\begin{lemma}
\label{lemma:std}
Let $\cT_{std} = (S_1,S_2,S_3)$ be the standard trisection of $S^{k}$.  Then each $S_i$ is diffeomorphic to $B^{k}$; each double intersection $S_i \cap S_j$ is diffeomorphic to $B^{k-1}$ and the central surface is diffeomorphic to $S^{k-2}$.

In particular, the standard trisection of $S^4$ is a $(0;0)$-trisection.
\end{lemma}

\begin{example}
\label{ex:unbalanced-S4}
$S^4$ also admits an unbalanced $(1;1,0,0)$-trisection.  It can be described by a trisection diagram on $T^2$ with attaching curves $\alpha_1,\alpha_2$ parallel and the third attaching curve $\alpha_3$ geometrically dual to both.  This implies that $Y_1 = H_1 \cup -H_2$ is homeomorphic to $S^1 \times S^2$ and bounds $S^1 \times B^3$, while $Y_2 = H_2 \cup - H_3$ and $Y_3 = H_3 \cup -H_1$ are homeomorphic to $S^3$ and bound $B^4$.  By permuting the indices, we also obtain unbalanced $(1;0,1,0)$ and $(1;0,0,1)$-trisections of $S^4$.
\end{example}

\subsection{Connected sum and stabilization}

\begin{definition}
\label{def:connected-sum}
Let $X$ and $Y$ be closed, smooth, oriented 4-manifolds with trisections $\cT_X,\cT_Y$ given by the decompositions:
\[X = X_1 \cup X_2 \cup X_3 \qquad \text{and} \qquad Y = Y_1 \cup Y_2 \cup Y_3\]
The {\it connected sum} of $\cT_X$ and $\cT_Y$ is the trisection $\cT_Z$ of $Z = X \# Y$ obtained by performing the connected sum along points $x \in \Sigma_X \subset X$ and $y \in \Sigma_Y \subset Y$ in the central surfaces of the trisections.  Each sector is precisely $Z_{\lambda} =  X_{\lambda} \natural Y_{\lambda}$.  If $X$ has a $(g;k_1,k_2,k_3)$-trisection and $Y$ has a $(h;l_1,l_2,l_3)$-trisection, then this construction gives a $(g+h;k_1+l_1,k_2+l_2,k_3+l_3)$-trisection of $Z$.
\end{definition}

Recall from Example \ref{ex:unbalanced-S4} that $S^4$ admits an (unbalanced) trisection of genus 1, where the sector $X_{\lambda} \cong S^1 \times B^3$ and the remaining sectors are 4-balls.  Let $\cT_{\lambda}$ denote this trisection.

\begin{definition}
A {\it sector-$\lambda$ stabilization} of a trisection $\cT$ of $X$ is the trisection $\cT'$ obtained by taking the connected sum of $\cT$ with the trisection $\cT_{\lambda}$ of $S^4$
\end{definition}

Stabilization increases the trisection genus by 1.  Note that in \cite[Definition 8]{Gay-Kirby}, stabilizations are assumed to be {\it balanced} and increase the genus by 3;  however this operation is equivalent to three unbalanced stabilizations, one in each sector.

\begin{remark}
There is an equivalence between trisection decompositions and handle decompositions \cite[Section 4]{Gay-Kirby}.  Under this equivalence, an unbalanced stabilization of a trisection translates to adding a pair of canceling handles.
\end{remark}

\begin{theorem}[\cite{Gay-Kirby}] Let $X$ be a closed, smooth, oriented 4-manifold.  Then
\begin{enumerate}
\item $X$ admits a trisection,
\item any two trisections of $X$ become ambient isotopic after sufficiently many stabilizations, and
\item $X$ is determined up to diffeomorphism by the spine of any trisection $\cT$.
\end{enumerate}
\end{theorem}

\begin{remark}
All of the results in \cite{Gay-Kirby} assume {\it balanced} trisections, meaning $k_1 = k_2 = k_3$, although in practice it is often useful to consider unbalanced trisections.  It is always possible to balance a trisection by stabilization, so the distinction is irrelevant here.
\end{remark}

\subsection{Relative trisections}

Let $Z_0 = B^4$ and $Z_k = \natural_k S^1 \times B^3$ and let $Y_k = \del Z_k$.  For any integers $r,s$ let $\Sigma_{s,t}$ denote the compact, oriented surface of genus $s$ and $t$ boundary components.  For any integers $r,s,t$, let $H_{r,s,t}$ denote a {\it compression body} cobordism between $\Sigma_{r,t}$ and $\Sigma_{s,t}$.  In particular, if $r = s$, then $H_{r,s,t} = \Sigma_{r,t} \times [0,1]$ and if $r > s$, then $H_{r,s,t}$ is obtained from $H_{r,r,t}$ by attaching $r - s$ 3-dimensional 2-handles.  

Let $g,p,b$ be nonnegative integers with $g \geq p$ and $b \geq 1$.  Following ?, there is a decomposition
\[Y_k = Y^-_{g,k;p,b} \cup Y^0_{g,k;p,b} \cup Y^+_{g,k;p,b}\]
where
\begin{align*}
Y^+_{g,k,p,b} & \cong H_{g,p,b} \\
Y^0_{g,k,p,b} & \cong \Sigma_{p,b} \times [0,1] \\
Y^-_{g,k,p,b} & \cong H_{g,p,b}
\end{align*}

\begin{definition}
Let $X$ be a compact, connected, oriented 4-manifold with connected boundary $Y$.  A $(g,k_1,k_2,k_3;p,b)$-{\it relative trisection} of $X$ is a decomposition
\[X = X_1 \cup X_2 \cup X_3\]
where
\begin{enumerate}
\item for each $\lambda = 1,2,3$, there is a diffeomorphism $\phi: X_{\lambda} \rightarrow Z_{k_{\lambda}} = \natural_{k_{\lambda}} S^1 \times B^3$ 
\item for each $\lambda = 1,2,3$ and taking indices mod 3, we have
\begin{align*}
\phi(X_{\lambda} \cap X_{\lambda - 1}) &= Y^-_{g,k_{\lambda},p,b} \\
\phi(X_{\lambda+1} \cap X_{\lambda}) &= Y^+_{g,k_{\lambda},p,b} \\
\phi(X_{\lambda} \cap \del X) &+ Y^0_{g,k_{\lambda},p,b}
\end{align*}
\end{enumerate}
\end{definition}

Analogously to the closed case, we define $H_{\lambda} = X_{\lambda} \cap X_{\lambda-1}$, which is a compression body.  

It follows from the construction (e.g. \cite[Lemma 11]{CGP1}) that a relative trisection induces an open book decomposition of $\del X$.  As in the closed case, we refer to $\Sigma = X_1 \cap X_2 \cap X_3$ as the {\it central surface}, the genus of $\Sigma$ as the {\it genus} of the trisection, and $H_! \cup H_2 \cup H_3$ as the {\it spine} of the trisection.

\begin{definition}
\label{def:std-tri-b4}
The {\it standard trisection} $\cT_{std}$ of $B^k$ for $k \geq 2$ is the decomposition $B^k = X_1 \cup X_2 \cup X_3$ where $X_{\lambda} = \pi^{-1}(D^s_{\lambda})$, where $\pi: \RR^k \rightarrow \RR^2$ is a coordinate projection, and $D = D^s_1 \cup D^s_2 \cup D^s_3$ is the standard trisection of the 2-disk.
\end{definition}

\begin{lemma}
Let $B^k = X_1 \cup X_2 \cup X_3$ be the standard trisection of $B^k$.  Then $X_{\lambda} \cap X_{\lambda+1}$ is diffeomorphic to $B^{k-1}$; the triple intersection $X_1 \cap X_2 \cap X_3$ is diffeomorphic to $B^{k-2}$; and $X_{\lambda} \cap \del B^k$ is diffeomorphic to $B^{k-1}$.

In particular, the standard trisection of $B^4$ is a $(0,0;0,1)$ relative trisection.
\end{lemma}

Note that if $X$ is a smooth 4-manifold with boundary with relative trisection $\cT_X$ and $Y$ is a closed, smooth 4-manifold with trisection $\cT_Y$, the {\it (interior) connected sum} of trisections $\cT_X \# \cT_Y$ of $X \# Y$ can be defined as in the closed case (Definition \ref{def:connected-sum}).

\begin{definition}
Let $X$ be smooth 4-manifold with boundary and let $\cT$ be a relative trisection of $X$. A {\it sector-$\lambda$ interior stabilization} of $\cT$ is the trisection obtained by taking the connected sum with the unbalanced, genus 1 trisection $\cT_{\lambda}$ of $S^4$.
\end{definition}

The main structural result for relative trisections is the following.

\begin{theorem}[\cite{Gay-Kirby}]
Let $X$ be a smooth, oriented 4-manifold with nonempty boundary.
\begin{enumerate}
\item For any fixed open book on the boundary $\del X$, there exists a relative trisection $\cT$  of $X$ inducing that open book deomposition.
\item Any two trisections of $X$ inducing the same open book on the boundary become ambient isotopic after a sequence of interior stabilizations.
\item the 4-manifold $X$ is determined up to diffeomorphism by the spine of a relative trisection.
\end{enumerate}
\end{theorem}

\subsection{Bridge position for surfaces}

Meier and Zupan generalized bridge position for links in a 3-manifold to bridge position for surfaces in a 4-manifold \cite{MZ1,MZ2}.  Meier has recently extended the theory of bridge trisections to properly embedded surfaces in 4-manifolds with boundary \cite{Meier-Relative}.  We will give a brief overview, focusing only on the relevant case $X \cong B^4$.

A properly embedded tangle $\tau$ in 3-manifold with boundary $M$ is {\it trivial} if the arcs of $\tau$ can be simultaneously isotoped to lie in $\del M$.  A {\it disk-tangle} $\cD \subset X$ in a 4-manifold is a collection of properly embedded disks.  A disk-tangle is {\it trivial} if the disks can be simultaneously isotoped to lie in the boundary. All of the trivial tangles and disk-tangles we consider will be of the following form.  Let $\Sigma$ be a compact surface with nonempty boundary and let $Q$ be a collection of points in $\Sigma$.  Then
\begin{enumerate}
    \item $Q \times [0,1]$ is a trival tangle in $\Sigma \times [0,1]$, and
    \item $Q \times [0,1]^2$ is a trivial disk-tangle in $\Sigma \times [0,1]^2$
\end{enumerate}

Recall that a link $L \subset Y$ is in {\it bridge position} with respect to a Heegaard splitting $Y = H_1 \cup_{\Sigma} H_2$ if each intersection $\tau_{\lambda} = L \cap H_{\lambda}$ is a trivial tangle.  Extending, we say that a properly embedded, closed surface $\cK$ in a closed 4-manifold $X$ is in {\it bridge position} with respect to a trisection $\cT$ of $X$ if
\begin{enumerate}
    \item $\cK$ intersects the central surface $\Sigma$ of the trisection transversely in $2b$ points,
    \item $\cK$ intersects each handlebody $H_{\lambda}$ along a trivial tangle $\tau_{\lambda}$.
    \item $\cK$ intersects each 4-dimensional sector $X_{\lambda}$ along a trivial disk tangle.
\end{enumerate}
The parameter $b$ is the {\it bridge index} of the surface $\cK$.

Now suppose that $\cK$ is a properly embedded surface in $B^4$, let $\cT$ be a relative trisection of $B^4$, and suppose that $\del \cK$ is braided with respect to the induced open book decomposition on $\del B^4$.  

\begin{definition}
We say that $\cK$ is in {\it relative bridge position} with respect to $\cT$ if
\begin{enumerate}
    \item $\cK$ intersects the central surface $\Sigma$ transversely in $2b + n$ points,
    \item $\cK$ intersects each handlebody $H_{\lambda}$ along a trivial tangle, and
    \item $\cK$ intersects each sector $X_{\lambda}$ along a trivial disk-tangle.
\end{enumerate}
If $n$ is the braid index of $\del \cK$, we say that $b$ is the {\it bridge index} of $\cK$.
\end{definition}

\subsection{Bridge perturbation}
\label{sub:bridge-perturbation}

For motivation, recall that a {\it bridge perturbation} of a link $L \subset Y$ in bridge position is a local modification that increases the bridge index within the same ambient isotopy class.  We can describe the move abstractly and in coordinates:

\begin{enumerate}
    \item {\bf Abstractly} Near a bridge point $x \in \Sigma$, we can choose a local neighborhood $U$ and coordinates $(x,y,)$ on $U$ such that (a) the Heegaard surface intersects $U$ along the $xy$-plane, and (b) the link $L$ intersects $U$ along the $z$-axis.  Now choose an arc $\delta$ in $U$, such that one endpoint lies on $L$ and the other on $\Sigma$.  This determines a finger-move isotopy of $L$ through $\Sigma$.
    \item {\bf In coordinates}.  In the coordinates on $U$ from above, we can isotope $L$ to be the linear graph $L = (t,0,t)$.  By an isotopy, we can further isotope
 the link to be the cubic graph $L' = (t,0,t^3 - \epsilon t)$.
\end{enumerate}

We can define an analogous procedure for surfaces in bridge position.  Following \cite[Definition 9.9]{Meier-Relative}, we first define a {\it finger-perturbation} abstractly.

\begin{definition}[\cite{Meier-Relative}]
Let $\cK$ be a surface in relative bridge position with respect to a relative trisection $\cT$.  Let $x$ be a bridge point and $U$ a neighborhood of $X$ such that $\cK$ intersects $U$ along a small disk.  The intersection of $\cK$ with $U \cap H_{\lambda}$ is an arc $\tau_{\lambda}$ with one endpoint at $x$.  A {\it sector-$\lambda$ finger perturbation} of $\cK$ is the surface $\cK'$ obtained by a finger-move isotopy along a small arc $\delta$ in $U \cap H_{\lambda+1}$, with one endpoint on $\tau_{\lambda+1}$ and the other in $\Sigma$.
\end{definition}

The finger move increases the bridge index by 1; it increases the number of components of $\cK \cap X_{\lambda+1}$ by one (since the sector $X_{\lambda+1}$ of the trisection is opposite to the handlebody $H_{\lambda}$); and it fixes the number of components of $\cK \cap X_{\lambda-1}$ and $\cK \cap X_{\lambda}$.

We can also describe this move visually in 3 dimensions by projecting away one dimension.  Let $\cK$ be a surface in bridge position with respect to a trisection $\cT$.  We can locally choose coordinates $(x,y,z)$ on $\RR^3$ near a bridge point $\cK \cap \Sigma$ such that
\[ \cK = \{ z = 0\}\]
and the trisection is given by pulling back the standard trisection of the $xy$-plane by the projection $\pi: \RR^3 \rightarrow \RR^2$.

A bridge perturbation consists of locally replacing $\cK$ by
\[\cK' = \{ x = z^3 - yz \} \]
Intuitively, this is accomplished by {\it pleating} the surface $\cK$ through the center of the trisection.

Restricted to $\cK$, the projection $\pi$ is a homeomorphism.  Restricted to $\cK'$, however, there is a critical local that projects to a semicubical cusp $(3t^2,-2t^3)$.  Over the interior of the cusp, the map is 3-to-1; over the exterior it is 1-to-1; at the origin there is a cusp singularity; and everywhere else is a fold singularity.

\subsection{Branched covers of surfaces in relative bridge position}

\begin{theorem}
\label{thrm:bridge-branch-tri}
Let $X$ be smooth 4-manifolds with boundary, let $\cK \subset X$ be a properly embedded surface, and let $\pi: \widetilde{X} \rightarrow X$ be an $n$-fold simple branched covering determined by a homomorphism $\rho: \pi_1(X \setminus \cK) \rightarrow S_n$.

\begin{enumerate}
    \item Suppose that $\cT$ is a relative trisection of $X$ and $\cK$ is in relative bridge position with respect to $\cT$.  Then $\widetilde{\cT} = \pi^{-1}(\cT)$ is a relative trisection of $\widetilde{X}$.
    \item Suppose that $\cK'$ isotopic to $\cK$ by a sector-$\lambda$ interior bridge perturbation.  Let $\widetilde{\cT},\widetilde{\cT}'$ be the induced relative trisections of $\widetilde{X}$.  Then $\widetilde{\cT}'$ is obtained from $\widetilde{\cT}$ by a sector-$(\lambda+1)$ interior stabilization.
    \end{enumerate}
\end{theorem}

\begin{proof}
Part (1) is the relative version of branched covering constructions of trisections , such as \cite[Theorem 4.1]{LMS}, \cite[Proposition 3.1]{LM-Rational}, and \cite[Proposition 13]{MZ2}.

An interior bridge perturbation increases the bridge index by 1 and the number of intersection points $\cK' \cap \Sigma$ by two.  Therefore, since the map $\pi$ is a simple branched covering, the Euler characteristic of the central surface $\widetilde{\Sigma}$ upstairs decreases by 2 or equivalently the genus increases by 1.  The number of components of $\cK' \cap X_{\lambda-1}$ and $\cK' \cap X_{\lambda}$ stays the same, while the number of components of $\cK' \cap X_{\lambda+1}$ increases by 1.  Therefore, the Euler characteristic of $\widetilde{X}_{\lambda-1},\widetilde{X}_{\lambda}$ stay the same; since they are 1-handlebodies this implies that their smooth topology stays the same.  However, the Euler characteristic of $\widetilde{X}_{\lambda+1}$ decreases by 1.  Thus $\widetilde{X}'_{\lambda+1}$ is obtained from $\widetilde{X}_{\lambda+1}$ by attaching a 1-handle.  In summary, this exactly corresponds to the trisection stabilization.
\end{proof}

\section{Stein trisections}
\label{sec:Stein}

A complex manifold $X$ is {\it Stein} if it admits a proper holomorphic embedding into some $\CC^N$. In particular, a Stein manifold admits an exhausting plurisubharmonic functions $f$, whose sublevel sets are pseudoconvex. Conversely, given a relative compact pseudoconvex domain $\Omega \subset X$, there exists an exhausting plurisubharmonic function such that $f^{-1}(-\infty,0)) = \Omega$.  We will generally restrict to relatively compact, pseudocovnex domains with smooth boundary in Stein surfaces, as these admit smooth trisections.

An {\it analytic polyhedron} in a complex manifold $X$ is a domain of the form
\[P = \{z \in X : |f_1(z)|,\dots,|f_k(z)| \leq 1 \}\]
where $f_1,\dots,f_k$ are holomorphic functions on $X$.  Analytic polyhedra are pseudoconvex and therefore Stein.

The main result of this section are constructions of Stein trisections of $B^4$.

\begin{theorem}
\label{thrm:Stein-B4}
Let $\cT$ be an interior stabilization of the standard relative trisection of $B^4$.  Then $\cT$ admits a Stein structure.
\end{theorem}




\subsection{Standard Stein trisection of $B^4$} As a first step, we construct a basic Stein trisection of the 4-ball.

Fix coordinates $z_j = x_j + i y_j$ on $\CC^2$.  For some real $M > 0$, define the subset
\[Q_M = \left\{ |x_j| \leq \frac{1}{M} \qquad |y_j| \leq M \right\}\]
Clearly, it is topologically equivalent to the standard $D^{2n}$.  Under the coordinate projection $\CC^2 \rightarrow \RR^2$ it maps to the unit square $[ -\frac{1}{n},\frac{1}{n} ] \times [ -\frac{1}{n},\frac{1}{n} ]$.  Consequently, it is the analytic polyhedron defined by the functions
\[ \left\{f_{i,\pm} = exp\left(\pm \left(z_i - \frac{1}{n}\right)\right) \right\}\]

Define functions
\[\phi_1 = x_2 \qquad \phi_2 =  - \sqrt{3} x_1 -x_2 \qquad \phi_3 = \sqrt{3}x_1 - x_2\]

\begin{theorem}
\label{thrm:b4-stein-standard}
The decomposition $Q_M = Z_1 \cup Z_2 \cup Z_3$ defined by setting
\[Z_{\lambda} = \{z \in Q_M : max(\phi_{\lambda},-\phi_{\lambda-1}) \leq 0\}\]
is a Stein trisection of $Q_M$.  In particular, it is equivalent to the standard relative trisection of $B^4 \cong Q_M$.

Furthermore, let $H_{\lambda} = Z_{\lambda} \cap Z_{\lambda -1}$ and $\Sigma = Z_1 \cap Z_2 \cap Z_3$.
\begin{enumerate}
    \item The central surface $\Sigma$ is a disk in the plane $i \RR^2 \subset \CC^2$
    \item The handlebody $H_{\lambda} = Z_{\lambda-1} \cap Z_{\lambda}$ lies in the hyperplane $\{\phi_{\lambda} = 0\}$ and is foliated by holomorphic disks.
\end{enumerate}
\end{theorem}

\begin{proof}
The functions $\phi_1,\phi_2,\phi_3$ are chosen to trisect the unit square $[ -\frac{1}{n},\frac{1}{n} ] \times [ -\frac{1}{n},\frac{1}{n} ] \subset \RR^2$ and the trisection of $Q_M$ is obtained by pulling this back via the projection $\CC^2 \rightarrow \RR^2$.  This construction is identical to Definition \ref{def:std-tri-b4}.

Statements (1) and (2) are immediate.  Each $\phi_{\lambda}$ is the real part of a linear holomorphic function $g_{\lambda} = a z_1 + b z_2$ for $a,b$ real.  The handlebody $H_{\lambda}$ is foliated by disks contained in complex lines form $g_{\lambda} = ic$ for some real constant $c$.
\end{proof}

\subsection{Holomorphic branched coverings and Stein trisections}

We can rephrase Theorem \ref{thrm:bridge-branch-tri} in the context of Stein trisections.  

\begin{theorem}
\label{thrm:hol-branched-pullback-tri}
Let $\widetilde{X},X$ be Stein surfaces and let $\pi: \widetilde{X} \rightarrow X$ be a holomorphic, simple branched covering map ramified over an embedded holomorphic curve $\cK \subset X$.  Let $\cT$ be a Stein trisection of $X$.  Suppose that $\cK$ is in relative bridge position with respect to $\cT$.  Then $\widetilde{\cT} = \pi^{-1}(\cT)$ is a Stein trisection of $\widetilde{X}$.
\end{theorem}

\begin{proof}
The smooth statement follows immediately from Theorem \ref{thrm:bridge-branch-tri}.  If the sector $X_{\lambda}$ is an analytic polyhedron defined by the holomorphic functions $f_1,f_2$, then the sector $\widetilde{X}_{\lambda} = \pi^{-1}(X_{\lambda})$ is an analytic polyhedron defined by the holomorphic functions $f_1 \circ \pi, f_2 \circ \pi$.  Therefore the decomposition is a Stein trisection.
\end{proof}

\subsection{An infinite family of Stein trisections of $B^4$}

Next, we construct an infinite family of Stein structures on the the standard, genus 0 relative trisection of $B^4$.

First, we construct the branch locus, which will be a collection of trivial disks in bridge position.  For some $\epsilon > 0$, define the following linear graph in $\CC^2$
\begin{align*}
    \Gamma_{\epsilon} &= (-i \epsilon z,z) \\
    &= (\epsilon y, -\epsilon x,x,y)
\end{align*}
for $z = x + i y$.  Moreover, for each $\theta \in [0,2 \pi]$ define the rotated line
\[\Gamma_{\epsilon,\theta} = \begin{bmatrix} \cos \theta & \sin \theta \\ -\sin \theta & \cos \theta \end{bmatrix} \cdot \Gamma_{\epsilon}\]

\begin{proposition}
\label{prop:infinite-family-Stein-B4}
Fix some $M,R$ such that $\frac{1}{M} \ll 1 \ll R \ll M$.  Let $\Theta = \{\theta_k\}$ for $k = 1,\dots,n$ be a collection of angles,  let 
\[ \cG_k = \Gamma_{\frac{1}{M},\theta_k} + \langle ikR, ikR \rangle \]
and let $\cG = \cup \cG_k$.
Then
\begin{enumerate}
\item The surface $\cG \cap Q_M$ is a collection of $n$ trivial disks in $B^4$
\item Each surface $\cG \cap Z_{\lambda}$ is a collection of $n$ trivial disks in $Z_{\lambda}$,
\item Each tangle $\cG \cap H_{\lambda}$ is a collection of $n$ trivial arcs in $H_{\lambda}$.
\item The intersection $\cG \cap \Sigma$ is a collection of $n$ points.
\end{enumerate}
In particular, $\cG \cap Q_M$ is in relative bridge position with respect to $\cT_0$.
\end{proposition}

\begin{proof}
Consider the intersection of 
\[ \Gamma_{0} + \langle ikR, ikR \rangle\]
with $Q_M$.  In coordinates $z = x + iy$, it is the surface
\[\left(\frac{1}{M} y, - \frac{1}{M} x + kR, x,y + kR\right)\]
where $|x| \leq \frac{1}{M}$ and $|y| \leq 1$.  In particular, since $R \gg 1$, we can assume that $G_k$ is disjoint from $G_{k'}$ for $k \neq k'$.  

To prove Statement (1), we will describe an explicit isotopy of $G_k$ into the boundary of $Q_M$.  Consider the family of surfaces
\[\left(\frac{1}{M} y,\frac{1}{M}(t-1)x + kR, x, (1 - t)y + tM + (1-t)kR \right)\]
for $t \in [0,1]$ and $|x|,|y| \leq \frac{1}{M}$.  This is precisely $G_k$ when $t = 0$ and is equal to 
\[\left(\frac{1}{M} y,kR, x,M\right)\]
when $t = 1$.  In particular, it is smoothly isotopic into the boundary, hence it is a trivial disk.  

To prove Statements (2)-(4), first note that the projection $\CC^2 \rightarrow \RR^2$ restricts to a diffeomorphism $G_k \rightarrow \RR^2$ on each component. Thus, the standard trisection of $\RR^2$ stratifies $G$.  In addition, the isotopy constructed for Part (1) also trivializes the disk tangles in each sector $Z_{\lambda}$ and the tangles in each handlebody $H_{\lambda}$.
\end{proof}

\begin{corollary}
\label{cor:infinite-stein}
Let $p: \widetilde{X} \rightarrow Q_M$ be the holomorphic branched covering determined by the map $\rho: \pi_1(Q_M \setminus G) \rightarrow S_{n+1}$ that sends the meridian of the $k^{\text{th}}$-disk to the transposition $(k \, k+1)$.  Then
\begin{enumerate}
\item $\widetilde{X}$ is diffeomorphic to $B^4$.
\item the pullback trisection $\widetilde{\cT} = p^{-1}(\cT_0)$, where $\cT_0$ is the trisection construction in Proposition \ref{thrm:b4-stein-standard} is diffeomorphic to the standard relative trisection of $B^4$.
\end{enumerate}
\end{corollary}

\begin{proof}
All the statements depend on the following observation.  Let $f(x) = x^{n+1} - \epsilon(n+1)x$ and consider the map $f: \CC \rightarrow \CC$.  This is an $(n+1)$-fold branched covering ramified at the $n^{\text{th}}$-roots of $\epsilon$.  The preimage $f^{-1}(\overline{\DD})$ is homeomorphic to a disk.  Moreover, this branched covering is simple and determined by a map $\rho': \pi_1(\DD \setminus \{\text{$n$ points}\}) \rightarrow S_{n+1}$ that agrees with $\rho$ up to conjugation.  

Recall that the trivial $n$-component tangle in $D^3$ can be thought of as taking $(D^2,\{\text{$n$ points}\})$ and crossing with the interval; similarly the trivial $n$-component disk tangle in $D^4$ is obtained by crossing $(D^2,\{\text{$n$ points}\})$ with two copies of the interval.  Therefore, crossing the branched covering $f$ with intervals, we get three branched coverings
\begin{align*}
    f:& \DD \rightarrow \DD \\
    f \times \text{id}: &\DD \times [0,1] \rightarrow \DD \times [0,1]\\
    f \times \text{id}: &\DD \times [0,1]^2 \rightarrow \DD \times [0,1]^2
\end{align*}
Up to diffeomorphism, these branched coverings are precisely the maps $p: \widetilde{X} \rightarrow Q_M$ as well as its restriction to each corresponding pair of strata of the trisections $\widetilde{\cT}$ and $\cT_0$.

In particular
\begin{align*}
    p^{-1}(\Sigma) & \cong D^2 &     p^{-1}(X_{\lambda}) & \cong D^4 \\
    p^{-1}(H_{\lambda}) & \cong D^3 &     p^{-1}(Q_M) & \cong D^4
\end{align*}
and $\widetilde{\cT}$ is a $(0,0;0,1)$-relative trisection of $D^4$, which is the standard trisection of the 4-ball.
\end{proof}

\subsection{Pleating}

Finally, we modify the branched coverings to induce stabilizations of the Stein trisections.  In particular, we can perform a bridge perturbation by an isotopy that pleats the branch locus.  This is accomplished by replacing the linear graphs of the previous subsection with cubic graphs.

Set
\begin{align*}
    \cC &= (i \epsilon (z^3 - z), z) \\
    &= (\epsilon(y^3 + (1 - x^2)y), \epsilon (x^3 - (1 + y^2)x, x,y)
\end{align*}
and let $\cC_{\theta}$ denotes its rotation by the angle $\theta$. 

\begin{proposition}
\label{prop:holomorphic-pleating}
Choose $\lambda \in \{1,2,3\}$ and define $\theta_{\lambda} = \frac{2 \pi}{3} \lambda$.  The surface 
\[\widetilde{\cG} = \cC_{\theta_{\lambda}}  + \langle 0, 1 + \epsilon' \rangle\]
is in relative bridge position with respect to $\cT_0$.  Moreover, it is a  sector-$\lambda$ bridge perturbation of $\Gamma_{\theta_{\lambda}}$.
\end{proposition}

\begin{proof}
We assume that $\lambda = 3$ so $\theta = 2\pi$ and the rotation is trivial.  The remaining cases follow by cyclic symmetry.

Projecting away the first imaginary coordinate, we can view the image of $\widetilde{\cG} \subset \RR \times \CC$, which is
\[(\epsilon (y^3 + (1 - x^2)y,x,y)\]
This is exactly the local model of a finger perturbation (Section \ref{sub:bridge-perturbation}).
\end{proof}

\subsection{Proof of main theorem}

We can now prove Theorem \ref{thrm:Stein-B4}.

\begin{proof}[Proof of Theorem \ref{thrm:Stein-B4}]
Let $\cT$ be a relative trisection of $B^4$ that is obtained from the standard trisection of $B^4$ by $n = n_1 + n_2 + n_3$ interior stabilizations, where $k_{\lambda}$ is the number of stabilizations in sector $\lambda$.

By Corollary \ref{cor:infinite-stein}, we have an $(n+1)$-fold holomorphic branched covering
\[p: \widetilde{X} \rightarrow Q_M\]
that pulls back the standard Stein trisection on $Q_M$ to a Stein structure on the standard relative trisection of $\widetilde{X} \cong B^4$.  This branched covering is ramified over $n$ disks in relative bridge position.  By Proposition \ref{prop:holomorphic-pleating}, we can modify the branch locus by bridge perturbations.  Specifically, for each $\lambda$, we can choose $n_{\lambda}$ disks and perform a sector-$\lambda$ bridge perturbation.

This yields a new branched covering
\[p': \widetilde{X}' \rightarrow Q_M\]
Since the branch locus is still holomorphic, Theorem \ref{thrm:hol-branched-pullback-tri} implies that $\pi'^{-1}(\cT)$ is a Stein trisection and Theorem \ref{thrm:bridge-branch-tri} implies that this smooth trisection is obtained by a $(n_1,n_2,n_3)$ interior stabilization of the standard relative trisection.
\end{proof}



\section{Reducible trisections}

Recall that trisections can be interpreted as 4-dimensional analogues of Heegaard splittings of 3-manifolds.  Meier, Schirmer and Zupan extended the well-known notion of {\it reducible Heegaard splittings} to define {\it reducible trisections} \cite{MSZ}.

\subsection{Reducible Heegaard splittings}

First, we recall the definition of reducibility in dimension 3 and state an easy result for parallelism.  

\begin{definition}
A Heegaard splitting $Y = H_{\alpha} \cup_{\Sigma} H_{\beta}$ is {\it reducible} if there is an essential simple closed curve $\delta \subset \Sigma$ that bounds a disk in both handlebodies.
\end{definition}

Note that if a Heegaard splitting is reducible, then the union of the two disks bounded by $\delta$ is an embedded 2-sphere bisected by the Heegaard splitting. Moreover, this 2-sphere is uniquely determined up to isotopy by the curve $\delta$. 

Conversely, given any 2-sphere we can find a Heegaard splitting that bisects it.

\begin{lemma}
\label{lemma:reducibleHeegaard}
Suppose that $S \subset Y$ is a separating 2-sphere.  There exists a reducible Heegaard splitting $Y = H_{\alpha} \cup_{\Sigma} H_{\beta}$ such that $S$ intersects $\Sigma$ along a separating, essential simple closed curve $\delta$.
\end{lemma}

\begin{proof}
The 2-sphere $S$ determines a connected sum decomposition $Y = Y_1 \#_S Y_2$.  We can construct the required Heegaard splitting of $Y$ by finding Heegaard splittings $Y_1 = H_{\alpha,1} \cup_{\Sigma_1} H_{\beta,1}$ and $Y_2 = H_{\alpha,2} \cup_{\Sigma_2} H_{\beta_2}$ and taking their connected sum.  Up to homeomorphism, the separating 2-sphere $S$ will intersect $\Sigma$ along the curve $\gamma$ that determines the connected sum $\Sigma = \Sigma_1 \#_{\delta} \Sigma_2$.
\end{proof}

\begin{remark}
Note that Haken's Lemma implies something stronger: if $S$ is an essential 2-sphere in $Y$, then {\it every} Heegaard splitting of $Y$ is reducible.  Currently, there is no 4-dimensional generalization of Haken's lemma to trisections of 4-manifolds.
\end{remark}

\subsection{Reducible trisections}

\begin{definition}
\label{def:std-B3}
Let $S \subset \#_k S^1 \times S^2$ be a smoothly embedded 2-sphere and let $B \subset \natural_k S^1 \times B^3$ be a properly embedded, smooth 3-ball with boundary $S$.  We say that $B$ is {\it standard} if $\#_k S^1 \times S^2 \setminus B$ has a handle decomposition consisting of only 0- and 1-handles.
\end{definition}

\begin{lemma}
\label{lemma:std-b3-unique}
Let $S \subset \#_k S^1 \times S^2$ be a smoothly embedded 2-sphere.
\begin{enumerate}
\item $S$ bounds a standard $B^3$ in $\natural_k S^1 \times B^3$.
\item Let $B_1,B_2$ be standard 3-balls with boundary $S$.  Then there exists a diffeomorphism of pairs $(\natural_k S^1 \times B^3,B_1) \cong (\natural_k S^1 \times B^3,B_2)$ that is the identity on the boundary.
\end{enumerate}
\end{lemma}

\begin{proof}
Identify $\#_k S^1 \times S^2$ with $\del \natural_k S^1 \times B^3$.  The latter has a handle decomposition consisting of one 0-handle and $k$ 1-handles.  By a result of Laudenbach \cite{Laudenbach}, we can perform a sequence of handle slides and isotopies so that $S$ is disjoint from the belt spheres of the 1-handles. In particular, we can assume $S$ lies in the boundary of the 0-handle.  Now we can add a canceling pair of 0/1-handles and isotope $S$ to agree with the belt sphere of the 1-handle.  The standard $B^3$ it bounds is the cocore of this 1-handle.

Part (2) is a restatement of the well-known theorem of Laudenbach and Po\'enaru \cite{LP} that the 3- and 4-handles of a smooth 4-manifold are attached uniquely, up to diffeomorphism.
\end{proof}

\begin{definition}[\cite{MSZ}]
\label{def:reducible}
A trisection $\cT$ of $X$ is {\it reducible} if there exists an essential simple closed curve $\delta$ in $\Sigma$ that bounds a disk $V_{\lambda}$ in each 3-dimensional handlebody $H_{\lambda}$ of the trisection. 
\end{definition}

Applying Lemma \ref{lemma:std-b3-unique} in each sector of a reducible trisection, we obtain the following proposition. 

\begin{proposition}
Let $\cT$ be a reducible trisection.  Let $\delta$ denote a reducing curve and let $V_{\lambda} \subset H_{\lambda}$ denote compressing disks bounded by $\delta$ Then
\begin{enumerate}
\item The union $S_{\lambda} = V_{\lambda} \cup V_{\lambda-1}$ is an embedded 2-sphere that bounds a standard 3-ball $B_{\lambda} \subset X_{\lambda}$.
\item The union $B_1 \cup B_2 \cup B_3$ is homeomorphic to $S^3$ and is exactly the standard trisection of $S^3$
\end{enumerate}
Moreover this 3-sphere is unique up to diffeomorphism.
\end{proposition}

Conversely, given any 3-sphere in a 4-manifold, we can find a reducible trisection that trisects it.

\begin{lemma}
\label{lemma:std-tri-3sphere}
Let $X$ be a closed, smooth, oriented 4-manifold and $S$ denote a smoothly embedded, separating $S^3$ in $X$.  There exists a reducible trisection $\cT$ of $X$ such that $S$ intersects:
\begin{enumerate}
\item the central surface along an essential separating simple closed curve $\delta$,
\item each 3-dimensional handlebody $H_{\lambda}$ along a single disk $D_{\lambda}$ with boundary $\delta$,
\item each 4-dimensional handlebody $X_{\lambda}$ along a single standard 3-ball $B_{\lambda}$ with boundary $D_{\lambda} \cup D_{\lambda+1}$.
\end{enumerate}
In particular, $\cT$ induces the standard trisection on $S$.
\end{lemma}

\begin{proof}
As in the proof of Lemma \ref{lemma:reducibleHeegaard}, we take the connected sum decomposition $X = X_1 \#_S X_2$, then trisect each factor and take the connected sum of trisections.
\end{proof}

\begin{definition}
A smoothly embedded, separating 3-sphere $S \subset X$ is {\it standard with respect to a trisection $\cT$} if it intersects the trisection as in Lemma \ref{lemma:std-tri-3sphere}.
\end{definition}

\begin{lemma}
Let $\cT$ be a reducible trisection
\begin{enumerate}
\item Every stabilization of $\cT$ is reducible.
\item If a 3-sphere $S$ is standard with respect to $\cT$, then up to isotopy we can assume it is standard with respect to every stabilization of $\cT$.
\end{enumerate}
\end{lemma}

\begin{proof}
Stabilization is well-defined and can be assumed to happen in the neighborhood of some point away from the curve $\delta$ (see the discussion after \cite[Definition 8]{Gay-Kirby}).
\end{proof}

Finally, standard 3-spheres are determined up to diffeomorphism by their intersection with the spine of the trisection.

\begin{proposition}
Suppose that $S_1,S_2$ are standard $S^3$'s in $X$ with respect to a trisection $\cT$ and suppose that $S_1 \cap \Sigma$ and $S_2 \cap \Sigma$ are isotopic in $\Sigma$ to the same simple closed curve $\delta$.  There is a diffeomorphism $\phi: X \rightarrow X$ fixing $\cT$ and such that $\phi(S_1) = S_2$.
\end{proposition}

\begin{proof}
By assumption, the curve $\delta$ bounds a disk $V_{\lambda}$ in $H_{\lambda}$ that is unique up to isotopy.  Thus, by an isotopy, we can assume that $S_1$ and $S_2$ have the same intersection with the spine of the trisection.  Let $A_{\lambda} = V_{\lambda} \cup V_{\lambda+1} \subset Y_{\lambda} = \del Z_{\lambda}$.  By assumption, $A_{\lambda}$ bounds a pair of standard 4-balls $B_{\lambda} = S_1 \cap Z_{\lambda}$ and $B'_{\lambda} = S_2 \cap Z_{\lambda}$; by an isotopy we can assume these agree in a tubular neighborhood of the spine.  But these are diffeomorphic by Lemma \ref{lemma:std-b3-unique}.  This diffeomorphism on the sectors can be glued to give a diffeomorphism of $X$ sending $S_1$ to $S_2$.
\end{proof}

\begin{remark}
Budney and Gabai found knotted 3-balls in $S^4$ \cite{BG} .  Thus, we have taken care to work up to {diffeomorphism} and not {ambient isotopy}.
\end{remark}

\section{Reconstructing homotopy 4-balls}
\label{sec:reconstruction}

In this section, we will generalize and let $X$ be any Stein surface, not just one diffeomorphic to $B^4$.  In particular, the reconstruction does not depend upon the underlying smooth topology of $X$.

\begin{theorem}
\label{thrm:pseudoconvex-reconstruction}
Let $X$ be a Stein surface, let $B$ be a homotopy 4-ball embedded in $X$ with smooth boundary $S$, and let $\cT = (X_1,X_2,X_3)$ be a Stein trisection of $X$ that is reducible with respect to $S$.  There exists a triple of pseudoconvex domains $Z_{\lambda} \subset X_{\lambda}$ such that $Z = Z_1 \cup Z_2 \cup Z_3$ is diffeomorphic to $B$.
\end{theorem}

\subsection{Setup}

To set up notation
\begin{align*}
\cT^0 & :=\text{the induced relative trisection decomposition $B = B_1 \cup B_2 \cup B_3$ }\\
\delta & :=\text{the reducing curve in the central surface $\Sigma$}\\
\Sigma^0 & :=\text{the genus $g'$ central surface of the relative trisection $\cT^0$, } \\
V_{\lambda} &:= \text{the compressing disk in $H_{\lambda}$ bounded by $\delta$}\\
S_{\lambda} & :=\text{the 2-sphere $V_{\lambda} \cup V_{\lambda+1}$ in $\del X_{\lambda}$}\\
H^0_{\lambda} & :=\text{the genus $g'$ handlebody bounded by $V_{\lambda} \cup \Sigma^0$}\\
\cS &:= \text{the spine $V_1 \cup V_2 \cup V_3$ of the reducing 3-sphere} \\
D_{\lambda} &:= \text{the standard 3-ball $X_{\lambda} \cap \del B$, which is bounded by $S_{\lambda}$}\\
B_{\lambda} &:= \text{the sector of $B$, which is bounded by $D_{\lambda} \cup H^0_{\lambda-1} \cup H^0_{\lambda}$}
\end{align*}

\subsection{Reconstruction}

Let $M_0$ denote the 3-sphere and $M_k$ denote the connected sum of $k$ copies of $S^1 \times S^2$.  The following lemma follows from the prime decomposition of 3-manifolds.

\begin{lemma}
\label{lemma:prime-decomp}
The 2-sphere $S_{\lambda}$ determines a connected sum splitting
\[\del X_{\lambda} = M_{k_{\lambda}} = M_{j_{1,\lambda}} \#_{S_{\lambda}} M_{j_{2,\lambda}}\]
where $k_{\lambda} = j_{1,\lambda} + j_{2,\lambda}$.
\end{lemma}

Next, we apply filling by holomorphic disks to replace $D_{\lambda}$.  We apply the main result of Bedford--Klingenberg.

\begin{theorem}[\cite{BK-filling}]
\label{thrm:BK-filling}
Let $X$ be a 2-dimensional Stein manifold, let $\Omega \subset \subset X$ be a strongly pseudoconvex domain with smooth boundary, and let $S$ be a smoothly embedded 2-sphere in $\del \Omega$. After a $C^2$-small perturbation of $S$, there is a smooth, properly embedded 3-ball $D \subset \Omega$ such that
\begin{enumerate}
    \item $S = \del D$
    \item $D$ is the disjoint union of holomorphic disks
    \item $S$ is foliated by the boundaries of the holomorphic disks
\end{enumerate}
In addition, the 3-ball $\overline{D}$ is the envelope of holomorphy of $S$.
\end{theorem}

\begin{proposition}
There exists a domain $Z_{\lambda} \subset X_{\lambda}$ such that
\begin{enumerate}
    \item $Z_{\lambda}$ and $X_{\lambda} \setminus Z_{\lambda}$ are pseudoconvex
    \item $\del Z_{\lambda} = \widetilde{D}_{\lambda} \cup H^0_{\lambda-1} \cup H^0_{\lambda}$, where $\widetilde{D}_{\lambda}$ is a 3-ball in $X_{\lambda}$ with boundary $S_{\lambda}$.
\end{enumerate}
\end{proposition}

\begin{proof}
The domain $X_{\lambda}$ is an analytic polyhedron. Consequently, it has a Stein neighborhood basis and we can arbitrarily approximate it from the outside by a Stein domain $\widetilde{X}_{\lambda}$ with smooth, strictly pseudoconvex boundary.  We can approximate $S_{\lambda}$ by a 2-sphere in $\del \widetilde{X}_{\lambda}$ and fill by 3-ball foliated by holomorphic disks (Theorem \ref{thrm:BK-filling}).  This 3-ball is pseudoconvex from both sides, so it decomposes $\widetilde{X}_{\lambda}$ into two pseudoconvex pieces.  We can then intersect these pieces with $X_{\lambda}$ and let $Z_{\lambda}$ denote the component with the specified boundary.
\end{proof}

The pseudoconvexity of $Z_{\lambda}$ and its complement determine their smooth topology.

\begin{proposition}
The domain $Z_{\lambda}$ is diffeomorphic to $\natural_{j_2} S^1 \times B^3$ and its complement $X_{\lambda} \setminus Z_{\lambda}$ is diffeomorphic to $\natural_{j_1} S^1 \times B^3$.
\end{proposition}

\begin{proof}
By Lemma \ref{lemma:reducibleHeegaard}, the boundary $\del Z_{\lambda}$ is homeomorphic to $\#_{j_{2,\lambda}} S^1 \times S^2$.  After smoothing the boundary, we have that $Z_{\lambda}$ is a Stein filling of its boundary.  The 3-manifold $\#_{j_{2,\lambda}} S^1 \times S^2$ admits a unique tight contact structure, which has a unique minimal Stein filling by $\natural_{j_{2,\lambda}} S^1 \times B^3$ (up to diffeomorphism) \cite[Chapter 16]{CE-12}.  Clearly $Z_{\lambda}$ is minimal since $X_{\lambda}$ is minimal.  The smooth topology of the complement is determined similarly.
\end{proof}

Applying Definition \ref{def:std-B3}, we obtain the following corollary.

\begin{corollary}
The 3-ball $\widetilde{D}_{\lambda}$ is a standard 3-ball.
\end{corollary}

Moreover, since $B^3 \cong D^2 \times [0,1]$ we also obtain the following corollary by definition.

\begin{corollary}
The decomposition $Z = Z_1 \cup Z_2 \cup Z_3$ is a relative trisection with spine $H^0_1 \cup H^0_2 \cup H^0_3$.
\end{corollary}

The final piece is the proof of Theorem \ref{thrm:pseudoconvex-reconstruction} is the following proposition.

\begin{proposition}
The 4-manifolds $B$ and $Z$ are diffeomorphic.
\end{proposition}

\begin{proof}
These two 4-manifolds admit relative trisections with identical spines.  Since the spine determines the trisected 4-manifold up to diffeomorphism, this implies that $B$ and $Z$ are diffeomorphic.
\end{proof}

\subsection{Shilov boundary criterion}

Let $\cS = V_1 \cup V_2 \cup V_3$ be the trisection spine.  

\begin{proof}[Proof of Theorem \ref{thrm:Shilov}]
Each $Z_{\lambda}$ is pseudoconvex.  Therefore we can choose a continuous plurisubharmonic function $G_{\lambda}: \Omega \rightarrow \RR_{\geq 0}$ such that $Z_{\lambda} = G^{-1}_{\lambda}(0)$.  We can glue these together to get an uppersemicontinuous function $G: \Omega \rightarrow \RR_{\geq 0}$ by setting
\[ G(z) := \begin{cases} \text{max}(G_1,G_2,G_3) & z \in \Sigma \\ \text{max}(G_{\lambda},G_{\lambda-1}) & z \in H_{\lambda} \\ G_{\lambda} & z \in \text{interior of } X_{\lambda}
\end{cases}\]
This function is upper semicontinuous and satisfies $Z = G^{-1}(0)$.  Since plurisubharmonicity is a local and open condition, it is plurisubharmonic on the complement of the spine.

Now suppose that $\cS$ is the Shilov boundary of $Z$ for some domain $Z \subset \Omega$.  For each $z \in \cS$, we can choose a holomorphic function $f_z$ on $\Omega$ that satisfies  $|f_z(z)| = |f_z|_Z = 1$.  Define the continuous, plurisubharmonic function
\[h = \text{max}_{z \in \cS} |f_z| - 1\]
By construction, we have that $h(Z) \subset [-1,0]$.  Define $\phi_c = \text{max}(G,Ch)$.  For $C \gg 0$, this is plurisubharmonic with $\phi^{-1}_C(0) = Z$.  In particular, $Z$ is pseudoconvex and therefore diffeomorphic to a standard 4-ball.
\end{proof}

\bibliographystyle{alpha}
\bibliography{References}


\end{document}